\documentclass[a4,11pt,oneside,reqno]{amsart}

\allowdisplaybreaks[1]

\makeatletter

\@addtoreset{equation}{section}

\DeclareMathOperator{\Rep}{Re}
\DeclareMathOperator{\Imp}{Im}
\DeclareMathOperator{\sign}{sign}
\DeclareMathOperator*{\Res}{Res}

\theoremstyle{plain}
\newtheorem*{Thm}{Theorem}
\newtheorem{Lem}{Lemma}[section]
\newtheorem{Prop}{Proposition}[section]
\newtheorem{Cor}{Corollary}[section]

\theoremstyle{definition}

\theoremstyle{remark}
\newtheorem{Rem}{Remark}[section]

\newcommand{\Cb}{\mathbb{C}}
\newcommand{\Qb}{\mathbb{Q}}
\newcommand{\Rb}{\mathbb{R}}
\newcommand{\Zb}{\mathbb{Z}}

\title[]{
A functional relation
for Tornheim's double zeta functions}

\author{Kazuhiro Onodera}

\address{Department of Mathematics, Tokyo Institute of Technology,
O-okayama, Meguro-ku, Tokyo
152-8551, Japan}

\email{onodera@math.titech.ac.jp}

\date{}

\thanks{This work was supported by
Grant-in-Aid for JSPS Fellows
(No.\ 22008809)}



\begin{document}
\baselineskip 15pt

\maketitle

\begin{abstract}
In this paper,
we generalize the partial fraction decomposition
which is fundamental in the theory of multiple zeta values,
and prove a relation between
Tornheim's double zeta functions of three complex variables.
As applications,
we give
new integral representations of several zeta functions,
an extension of the parity result to the whole domain of convergence,
concrete expressions of Tornheim's double zeta function
at non-positive integers
and some results for the behavior of
a certain Witten's zeta function
at each integer.
As an appendix,
we show a functional equation
for Euler's double zeta function.
\end{abstract}
\section{Introduction}
Tornheim's double zeta function is defined as
\[
\zeta(s,t;u)=
\sum_{m,n=1}^\infty \frac1{m^s n^t (m+n)^u}
\]
for $(s,t,u)\in\Cb^3$
with $\Rep(s+u)>1$, $\Rep(t+u)>1$ and $\Rep(s+t+u)>2$.
It is known, by Matsumoto \cite[Theorem 1]{Matsumoto02},
that $\zeta(s,t;u)$ can be meromorphically continued to
the whole space $\Cb^3$,
and its singularities are located on
the subsets of $\Cb^3$ defined by one of
the equations
$s+u=1-l$, $t+u=1-l$ ($l=0,1,2,\dotsc$)
and $s+t+u=2$.
This function can be regarded as a generalization
of some well-known zeta functions:
the product of two Riemann zeta functions
$\zeta(s)\zeta(t)=\zeta(s,t;0)$,
the Euler double zeta function $\zeta(u,t)=\zeta(0,t;u)$
and the $SU(3)$-type Witten zeta function
$\zeta_{SU(3)}(s)=2^s \zeta(s,s;s)$.
Euler and Tornheim \cite{Tornheim50}, and many people
gave a lot of relations between the values $\zeta(s,t;u)$
for triples $(s,t,u)$ of non-negative integers
on the domain of convergence,
but
little relation as functions of complex variables
has been found.
As an exception,
Tsumura \cite[Theorem 4.5]{Tsumura07}
represented explicitly the function
\begin{equation}
\label{eq:def of Z}
Z(s,t;u)=\zeta(s,t;u)+\cos(\pi t)\zeta(t,u;s)
+\cos(\pi s)\zeta(u,s;t)
\end{equation}
in terms of the Riemann zeta function,
when $s,t\in\Zb_{\ge 0}$, $t\ge 2$ and $u\in\Cb$,
except for singularities.
Afterward Nakamura \cite[Theorem 1.2]{Nakamura06}
gave a simpler version:
for $s,t\in\Zb_{\ge 1}$ and $u\in\Cb$
except for the singular points,
\begin{align}
Z(s,t;u)
&=
2\sum_{h=0}^{[s/2]}
\binom{s+t-2h-1}{s-2h}
\zeta(2h)\zeta(s+t+u-2h)
\nonumber\\
&\quad
+2\sum_{k=0}^{[t/2]}
\binom{s+t-2k-1}{t-2k}
\zeta(2k)\zeta(s+t+u-2k),
\label{eq:Nakamura's result}
\end{align}
where $[x]$ for $x\in\Rb$
denotes the greatest integer not exceeding $x$.
This result seems really fascinating
because it contains most of the known relations
between the values $\zeta(s,t;u)$ for $s,t,u\in\Zb_{\ge 1}$
and the Riemann zeta values
(see \cite[\S 3]{Nakamura06}).
The aim of this paper is to
generalize it to
a relation between
Tornheim's double zeta functions of three complex variables.
Our main result is

\begin{Thm}
The following relation holds
on the whole space $\Cb^3$
except for the singular points
of both sides:
\[
Z(s,t;u)=A(s,t;u)+A(t,s;u),
\]
where,
for $(s,t,u)\in\Cb^3$ with $s,-t, 1-t-u\ne 0,1,2,\dotsc$,
\begin{equation}
\label{thm-eq:definition of A}
A(s,t;u)
=
\frac{\sin(\pi s)}{2\pi i}
\int_L
\cot\left(\frac{\pi(s-\eta)}2\right)
\frac{\Gamma(t+\eta)\Gamma(-\eta)}{\Gamma(t)}
\zeta(s-\eta)\zeta(t+u+\eta)d\eta.
\end{equation}
Here,
the contour $L$ is
a line from $-i\infty$ to $i\infty$
indented in such a manner as to
separate the poles at
$\eta=s-2n,-t-n,1-t-u~(n=0,1,2,\dotsc)$
from the poles at $\eta=0,1,2,\dotsc$.
\end{Thm}

\begin{Rem}
The singularities of $A(s,t;u)$
are located only on the subsets of $\Cb^3$
defined by the equations
$t+u=1-l~(l=0,1,2,\dotsc)$.
This can be easily seen by shifting the contour $L$
as follows.
Let $K$ be a non-negative integer.
If $\Rep(s)<K+1/2$, $-K-1/2<\Rep(t)$ and $-K+1/2<\Rep(t+u)$,
then
\begin{align}
A(s,t;u)
&=
2\sum_{k=0}^K
\frac{(t)_k}{k!}
\cos^2\left(\frac{\pi(s-k)}2\right)
\zeta(s-k)\zeta(t+u+k)
\nonumber\\
&\quad
+\frac{\sin(\pi s)}{2\pi i}
\int_{L_K}\cot\left(\frac{\pi (s-\eta)}2\right)
\frac{\Gamma(t+\eta)\Gamma(-\eta)}{\Gamma(t)}
\zeta(s-\eta)\zeta(t+u+\eta) d\eta,
\label{eq:A-function shifted}
\end{align}
where
$(t)_k=\Gamma(t+k)/\Gamma(t)$ and
$L_K$ describes the vertical line
from $K+1/2-i\infty$ to $K+1/2+i\infty$.
Also, from this, it is clear that our functional relation is
a generalization of (\ref{eq:Nakamura's result})
(see (\ref{lem-eq:A(a,t;u)})).
\end{Rem}

\begin{Rem}
Since
\[
\begin{pmatrix}
Z(s,t;u)\\
Z(t,u;s)\\
Z(u,s;t)
\end{pmatrix}
=\begin{pmatrix}
1 & \cos(\pi t) & \cos(\pi s)\\
\cos(\pi t) & 1 & \cos(\pi u)\\
\cos(\pi s) & \cos(\pi u) & 1
\end{pmatrix}
\begin{pmatrix}
\zeta(s,t;u)\\
\zeta(t,u;s)\\
\zeta(u,s;t)
\end{pmatrix},
\]
we can write
$\zeta(s,t;u)$
in terms of the $Z$-function as
\begin{align}
\Delta(s,t,u)\zeta(s,t;u)
&=(1-\cos^2(\pi u))Z(s,t;u)
\nonumber\\
&\quad
+(\cos(\pi s)\cos(\pi u)-\cos(\pi t))Z(t,u;s)
\nonumber\\
&\quad
+(\cos(\pi t)\cos(\pi u)-\cos(\pi s))Z(u,s;t),
\label{eq:representation of zeta by Z-function}
\end{align}
where
$\Delta(s,t,u)
=1-\cos^2(\pi s)-\cos^2(\pi t)-\cos^2(\pi u)
+2\cos(\pi s)\cos(\pi t)\cos(\pi u)$,
and so Theorem
gives a new integral representation of $\zeta(s,t;u)$.
Some special cases will be displayed
in Proposition \ref{prop:representation of zetas}.
\end{Rem}

In this paper,
to prove Theorem,
we employ Li's method in \cite{Lie_arxiv}
which gave a simple proof of (\ref{eq:Nakamura's result}).
In \S 2, we will generalize some partial fraction decompositions
used there
to a usable form in our case.
We will give a proof of Theorem in \S 3
and exhibit its applications in \S 4.
In Appendix,
we will show a functional equation
for Euler's double zeta function.
\section{Generalized partial fraction decomposition}
The following partial fraction decomposition
plays a fundamental role in the theory of multiple zeta values:
for two independent variables $p,q$
and two positive integers $s,t$,
\begin{equation}
\label{eq:classical PFD}
\frac1{p^sq^t}
=\sum_{h=0}^{s-1}
\frac{(t)_h}{h!}\frac1{p^{s-h}(p+q)^{t+h}}
+\sum_{k=0}^{t-1}
\frac{(s)_k}{k!}\frac1{q^{t-k}(p+q)^{s+k}}.
\end{equation}
In this section,
we will formulate
two partial fraction decompositions
in the case of $s,t$ being complex numbers.

\begin{Lem}
\label{lem:GPFD1}
Let $p,q$ be positive real numbers
and let $s,t$ be complex numbers
whose real parts are positive.
If $s,t\ne 1,2,\dotsc$, then
\begin{equation}
\label{lem-eq1:GPFD1}
\frac{\Gamma(s)\Gamma(t)}{p^s q^t}
=I(s,t;p,r)+I(t,s;q,r),
\end{equation}
where $r=p+q$ and
\begin{equation}
\label{lem-eq2:GPFD1}
I(s,t;p,r)
=\frac1{2\pi i}
\int_{L_{s,t}}
\frac{\Gamma(1-s+\eta)\Gamma(-\eta)}{\Gamma(1-s)}
\frac{\Gamma(s-\eta)}{p^{s-\eta}}
\frac{\Gamma(t+\eta)}{r^{t+\eta}}d\eta.
\end{equation}
Here, the contour $L_{s,t}$ is a line from $-i\infty$ to $i\infty$
indented in such a manner as to separate
the points at $\eta=s-1-m,-t-m~(m=0,1,2,\dotsc)$
from the points at $\eta=s+n,n~(n=0,1,2,\dotsc)$.
\end{Lem}

\begin{proof}
From the usual integral representation of the gamma function,
it follows that
\[
I(s,t;p,r)
=\frac1{2\pi i}\int_{L_{s,t}}
\frac{\Gamma(1-s+\eta)\Gamma(-\eta)}{\Gamma(1-s)}
\left(\iint_{(\Rb_{>0})^2}
e^{-p\mu-r\nu} \mu^{s-\eta-1} \nu^{t+\eta-1}
d\mu d\nu
\right)
d\eta.
\]
By a suitable choice of $L_{s,t}$,
it can be shown that
the order of the integrations can be interchanged.
Hence, we see
\[
I(s,t;p,r)
=
\iint_{(\Rb_{>0})^2}
e^{-p\mu-r\nu} \mu^{s-1}\nu^{t-1}
\left(\frac1{2\pi i}\int_{L_{s,t}}
\frac{\Gamma(1-s+\eta)\Gamma(-\eta)}{\Gamma(1-s)}
(\nu/\mu)^\eta
d\eta\right)d\mu d\nu.
\]
Since the innermost integral is $(1+\nu/\mu)^{s-1}$
(see \cite[\S14.51, Corollary]{Whittaker-Watson}),
we obtain
\begin{align*}
I(s,t;p,r)
&=
\iint_{(\Rb_{>0})^2}
e^{-p\mu-r\nu} (\mu+\nu)^{s-1}\nu^{t-1}d\mu d\nu\\
&=\iint_{0<\nu<\mu} e^{-p\mu-q\nu} \mu^{s-1}\nu^{t-1}
d\mu d\nu.
\end{align*}
We note that
\[
I(t,s;q,r)
=\iint_{0<\mu<\nu} e^{-p\mu-q\nu} \mu^{s-1}\nu^{t-1}
d\mu d\nu,
\]
and so
the right hand side of (\ref{lem-eq1:GPFD1})
is equal to
\[
\iint_{(\Rb_{>0})^2}
e^{-p\mu-q\nu} \mu^{s-1}\nu^{t-1}
d\mu d\nu
=\frac{\Gamma(s)\Gamma(t)}{p^s q^t}.
\]
This is the desired result.
\end{proof}


\begin{Lem}
\label{lem:GPFD2}
Let $p,q,s,t$ be as in Lemma $\ref{lem:GPFD1}$.
If $p<q$ and $s,t\ne 1,2,\dotsc$, then
\begin{equation}
\label{lem-eq1:GPFD2}
\frac{\cos(\pi s)\Gamma(s)\Gamma(t)}{p^s q^t}
=J(s,t;p,q-p)+I(t,s;q,q-p),
\end{equation}
where
\[
J(s,t;p,q-p)
=\frac1{2\pi i}
\int_{L_{s,t}}
\frac{\Gamma(1-s+\eta)\Gamma(-\eta)}{\Gamma(1-s)}
\frac{\cos(\pi(s-\eta))\Gamma(s-\eta)}{p^{s-\eta}}
\frac{\Gamma(t+\eta)}{(q-p)^{t+\eta}}d\eta.
\]
\end{Lem}

\begin{proof}
For any $p,r\in\Cb^\times$,
the integrand in (\ref{lem-eq2:GPFD1})
is
\[
\ll {|\eta|}^{\Rep(t)-1} e^{-2\pi |\eta|}
e^{\Imp(\eta)(\arg r -\arg p)}
|p^{-s}||r^{-t}|
\]
as $\eta\rightarrow \pm i\infty$ on $L_{s,t}$,
where the implied constant does not depend on $p,r,\eta$.
This estimate ensures that, for any fixed $q\in\Rb_{>0}$,
the right hand side of (\ref{lem-eq1:GPFD1}) is
continued to
$\Cb\setminus\{\pm i\Rb_{\ge0}\cup (-q\pm i\Rb_{\ge 0})\}$
as a holomorphic function in $p$,
where the double-signs correspond,
and hence,
if $0<p<q$, then
\[
\frac{e^{\pm \pi i s}\Gamma(s)\Gamma(t)}{p^s q^t}
=\frac{\Gamma(s)\Gamma(t)}{(-p)^s q^t}
=I(s,t;-p,q-p)+I(t,s;q,q-p)
\]
and
\[
I(s,t;-p,q-p)
=\frac1{2\pi i}
\int_{L_{s,t}}
\frac{\Gamma(1-s+\eta)\Gamma(-\eta)}{\Gamma(1-s)}
\frac{e^{\pm \pi i(s-\eta)}\Gamma(s-\eta)}{p^{s-\eta}}
\frac{\Gamma(t+\eta)}{(q-p)^{t+\eta}}d\eta.
\]
Thus, we obtain Lemma \ref{lem:GPFD2}.
\end{proof}

\section{Proof of Theorem}
For simplicity of description,
we suppose that $\Rep(s),\Rep(t),\Rep(u)>2$,
$s,t\ne 3,4,5,\dotsc$
and the contour $L_{s,t}$
always satisfies the condition $-1/2\le \Rep(\eta)\le \Rep(s)-1/2$
for all $\eta\in L_{s,t}$.
We first evaluate
\[
\zeta(s,t;u)
=\sum_{m,n=1}^\infty \frac1{m^s n^t (m+n)^u}.
\]
Applying Lemma \ref{lem:GPFD1}
with $(p,q)=(m,n)$,
we see
\[
\zeta(s,t;u)
=X(s,t;u)+X(t,s;u),
\]
where
\[
X(s,t;u)
=\frac1{2\pi i}\sum_{m,n=1}^\infty
\int_{L_{s,t}}
\frac{\Gamma(s,t;\eta)}{m^{s-\eta}(m+n)^{t+u+\eta}}d\eta
\]
and
\[
\Gamma(s,t;\eta)=
\frac{\Gamma(1-s+\eta)\Gamma(-\eta)
\Gamma(s-\eta)\Gamma(t+\eta)}
{\Gamma(1-s)\Gamma(s)\Gamma(t)}.
\]
From the condition of $L_{s,t}$,
it follows that the order of summation and integration
of $X(s,t;u)$
can be interchanged.
As a result, we have
\begin{align}
X(s,t;u)
&=\frac1{2\pi i}
\int_{L_{s,t}}
\Gamma(s,t;\eta)\zeta(s-\eta,0;t+u+\eta)d\eta
\nonumber \\
&=\frac{\Gamma(s+t-1)}{\Gamma(s)\Gamma(t)}
\zeta(1,0;s+t+u-1)
\nonumber \\
&\quad
+\frac1{2\pi i}
\int_{L_{s-1,t}}
\Gamma(s,t;\eta)\zeta(s-\eta,0;t+u+\eta)d\eta.
\label{pf-eq:X(s,t;u)}
\end{align}

We next treat
$\cos(\pi t)\zeta(t,u;s)+\cos(\pi s)\zeta(u,s;t)$.
Set
\[
a_{m,n}(s,t;u)
=
\frac{\cos(\pi t)}{m^t n^u (m+n)^s}
+\frac{\cos(\pi s)}{n^u m^s (m+n)^t}
\]
for $m,n\in\Zb_{\ge 1}$.
Applying Lemma \ref{lem:GPFD2}
to each term,
we obtain
\[
a_{m,n}(s,t;u)
=b_{m,n}(s,t;u)+b_{m,n}(t,s;u),
\]
where
\begin{align*}
b_{m,n}(s,t;u)
&=
\frac{I(s,t;m+n,n)+J(s,t;m,n)}{\Gamma(s)\Gamma(t) n^u}\\
&=-
\frac{\Gamma(s+t-1)}{\Gamma(s)\Gamma(t)}
\frac1{n^{s+t+u-1}}
\left(\frac1m-\frac1{m+n}\right)\\
&\quad
+\frac1{2\pi i}
\int_{L_{s-1,t}}\Gamma(s,t;\eta)
\left(
\frac1{n^{t+u+\eta}(m+n)^{s-\eta}}
+\frac{\cos(\pi(s-\eta))}
{m^{s-\eta}n^{t+u+\eta}}
\right)d\eta.
\end{align*}
Put $Y(s,t;u)=\sum_{m,n=1}^\infty b_{m,n}(s,t;u)$.
Then, it is easily seen that
\[
\cos(\pi t)\zeta(t,u;s)+\cos(\pi s)\zeta(u,s;t)
=Y(s,t;u)+Y(t,s;u),
\]
and that
\begin{align}
Y(s,t;u)
&=
-\frac{\Gamma(s+t-1)}{\Gamma(s)\Gamma(t)}
\Big(
\zeta(1,0;s+t+u-1)+\zeta(s+t+u)
\Big)
\nonumber\\
&\quad
+\frac1{2\pi i}
\int_{L_{s-1,t}}
\Gamma(s,t;\eta)
\Big\{\zeta(t+u+\eta,0;s-\eta)
\nonumber\\
&\hspace{30mm}
+\cos(\pi(s-\eta))\zeta(s-\eta)\zeta(t+u+\eta)\Big\}
d\eta.
\label{pf-eq:Y(s,t;u)}
\end{align}

Combining (\ref{pf-eq:X(s,t;u)}) and (\ref{pf-eq:Y(s,t;u)}),
we have
\begin{align}
\lefteqn{X(s,t;u)+Y(s,t;u)}
\quad&
\nonumber\\
&=-\frac{\Gamma(s+t-1)}{\Gamma(s)\Gamma(t)}
\zeta(s+t+u)
\nonumber\\
&\quad
+\frac1{2\pi i}
\int_{L_{s-1,t}}\Gamma(s,t;\eta)
\{1+\cos(\pi (s-\eta))\}\zeta(s-\eta)\zeta(t+u+\eta)
d\eta
\nonumber\\
&\quad
-\frac1{2\pi i}\int_{L_{s-1,t}}
\Gamma(s,t;\eta)d\eta\,
\zeta(s+t+u)
\label{pf-eq:X+Y}
\end{align}
because generally
\begin{equation}
\label{pf-eq:harmonic product formula}
\zeta(s,0;t)
+\zeta(t,0;s)
=\zeta(s)\zeta(t)
-\zeta(s+t).
\end{equation}
The second term on the right hand side
of (\ref{pf-eq:X+Y})
becomes
\[
\frac{\Gamma(s+t)}{s\Gamma(s)\Gamma(t)}
\zeta(s+t+u)
+A(s,t;u)
\]
by shifting the contour to $L_{s+1,t}$,
and the third term is
\begin{align*}
&=
\frac{\Gamma(s+t-1)}{\Gamma(s)\Gamma(t)}
\zeta(s+t+u)
-\frac1{2\pi i}\int_{L_{s,t}}
\Gamma(s,t;\eta)d\eta\,
\zeta(s+t+u)\\
&=\frac{\Gamma(s+t-1)}{\Gamma(s)\Gamma(t)}
\zeta(s+t+u)-
\frac{\Gamma(s+t)}{t\Gamma(s)\Gamma(t)}
\zeta(s+t+u)
\end{align*}
by Barnes' lemma
(see \cite[14.52]{Whittaker-Watson}).
Hence,
\[
X(s,t;u)+Y(s,t;u)
=\left(\frac1s-\frac1t\right)
\frac{\Gamma(s+t)}{\Gamma(s)\Gamma(t)}\zeta(s+t+u)
+A(s,t;u).
\]
Thus, we have
\begin{align*}
Z(s,t;u)
&=X(s,t;u)+X(t,s;u)+Y(s,t;u)+Y(t,s;u)\\
&=A(s,t;u)+A(t,s;u),
\end{align*}
when $\Rep(s),\Rep(t),\Rep(u)>2$
and $s,t\ne 3,4,5,\dotsc$.
By the theory of analytic continuation,
the proof of Theorem is completed.

\begin{Rem}
We have used Li's method in this section,
but it is possible to prove Theorem
by Nakamura's original method.
Indeed,
all his argument is valid here
except for the property for the Bernoulli polynomial
\cite[(2,7)]{Nakamura06},
which can be proved
by the partial fractional decomposition
(\ref{eq:classical PFD})
in a similar way to Eisenstein's proof
for addition formulas of the trigonometric functions
(see \cite[Chapter II]{Weil76} or \cite[\S 2.1]{Sczech92}).
And so it can be reformulated to an available form in our case
by (a slightly generalized) Lemma \ref{lem:GPFD1}.
\end{Rem}

\section{Application}
In this section,
we will find some new results from Theorem.
The following each proposition
can be proved independently of the others.
However, the next lemma
seems to be useful for some applications,
and so we first state it in order to access quickly.

\begin{Lem}
\label{lem:evaluation of A}
Let $a$ be an integer
and $b,c$ be non-negative integers.
Set
\begin{equation}
\label{lem-eq:definition of F}
F(s,t;c)
=\sum_{k=0}^c \binom{c}{k}\zeta(s-k)\zeta(t-c+k)
\end{equation}
for $s,t\in\Cb$.
Then we have the followings:

$(1)$
For $t,u\in\Cb$ with $t+u\ne 1-l~(l=0,1,2,\dotsc)$,
\begin{equation}
\label{lem-eq:A(a,t;u)}
A(a,t;u)
=2\sum_{k=0}^{[a/2]}
\binom{t+a-2k-1}{a-2k}
\zeta(2k)\zeta(t+u+a-2k),
\end{equation}
where the value of any empty sum is defined to be $0$.

$(2)$
For $s,u\in\Cb$ with $u\ne b+1-l~(l=0,1,2,\dotsc)$,
\[
A(s,-b;u)
=
\sum_{k=0}^b
\binom{b}{k}
(\cos(\pi s)+(-1)^k)
\zeta(s-k)\zeta(u-b+k).
\]

$(3)$
For any $s\in\Cb$,
\[
\lim_{u\rightarrow -c}
A(s,-b;u)
=(\cos(\pi s)-(-1)^{b+c})
F(s,-c;b)
+\delta_{c0}(-1)^{b+1}\zeta(s-b),
\]
where
$\delta_{ij}$ denotes the Kronecker symbol.

$(4)$
For any $s\in\Cb$,
\begin{align*}
\lefteqn{\lim_{t\rightarrow -b}
A(s,t;-c)}
\quad &\\
&=(\cos(\pi s)-(-1)^{b+c})
\left(F(s,-c;b)
+\frac{(-1)^{c+1} b! c!}{(b+c+1)!}
\zeta(s-b-c-1)\right)\\
&\quad
+\delta_{c0}(-1)^{b+1}\zeta(s-b).
\end{align*}
\end{Lem}

\begin{proof}
(1),(2)
The formulas follow 
immediately from (\ref{eq:A-function shifted}).

(3)
We apply (2) to get
\begin{align*}
\lim_{u\rightarrow -c}A(s,-b;u)
&=(\cos(\pi s)-(-1)^{b+c})
\sum_{k=0}^{b-1}
\binom{b}{k}\zeta(s-k)\zeta(-b-c+k)\\
&\quad
+(\cos(\pi s)+(-1)^b)\zeta(s-b)\zeta(-c),
\end{align*}
where we have used the fact that
$\zeta(-b-c+k)=0$
if $0\le k\le b-1$ and $k\equiv b+c \pmod 2$.
Thus, by simple calculation, we obtain the result.

(4)
The result follows in a similar way to the above.
\end{proof}

We give integral representations of several zeta functions.

\begin{Prop}
\label{prop:representation of zetas}
$(1)$
The Euler double zeta function $\zeta(s,t)$
has the following representation:
\begin{align}
\lefteqn{(\cos(\pi t)-\cos(\pi s))\zeta(s,t)}
\quad &
\nonumber\\
&=A(s,t;0)+A(t,s;0)-(1+\cos(\pi s))\zeta(s)\zeta(t)
+\cos(\pi s)\zeta(s+t).
\label{prop-eq:representation of Euler double zeta}
\end{align}

$(2)$
Let $n$ be a non-negative integer. Then,
\begin{align}
\lefteqn{(1+\cos(\pi s))\zeta(s)\zeta(s+2n)}
\quad &
\nonumber\\
&=A(s,s+2n;0)+A(s+2n,s;0)
+\cos(\pi s)\zeta(2s+2n)
\label{prop-eq:product of zeta 1}
\end{align}
and
\begin{align*}
\lefteqn{(1+\cos(\pi s))\zeta(s)\zeta(-s+2n)}
\quad &\\
&=A(s,-s+2n;0)+A(-s+2n,s;0)
+\cos(\pi s)\zeta(2n)
+\delta_n(s),
\end{align*}
where
\[
\delta_n(s)=
\begin{cases}
-\pi s \sin(\pi s)/12 & \text{if $n=0$},\\
-\pi \sin(\pi s)/(s-1) & \text{if $n=1$},\\
0 & \text{otherwise}.
\end{cases}
\]
In particular, we obtain
\begin{align*}
\lefteqn{(1+\cos(\pi s))\zeta(s)^2-\cos(\pi s)\zeta(2s)
=2A(s,s;0)}
\quad &\\
&=\frac{2\sin(\pi s)}{2\pi i}
\int_{L}
\cot\left(\frac{\pi(s-\eta)}2\right)
\frac{\Gamma(s+\eta)\Gamma(-\eta)}{\Gamma(s)}
\zeta(s-\eta)\zeta(s+\eta)d\eta.
\end{align*}

$(3)$
The Witten zeta function of $SU(3)$
can be written as
\begin{align}
\lefteqn{2^{-s-1}(1+2\cos(\pi s))\zeta_{SU(3)}(s)
=A(s,s;s)}
\quad &
\nonumber\\
&=
\frac{\sin(\pi s)}{2\pi i}
\int_{L}
\cot\left(\frac{\pi(s-\eta)}2\right)
\frac{\Gamma(s+\eta)\Gamma(-\eta)}{\Gamma(s)}
\zeta(s-\eta)\zeta(2s+\eta)d\eta.
\label{prop-eq:integral representation of Witten zeta}
\end{align}
\end{Prop}

\begin{Rem}
We can regard
(\ref{prop-eq:product of zeta 1})
as a generalization of the formula
\begin{align*}
\lefteqn{\zeta(2l)\zeta(2m)-\frac12 \zeta(2l+2m)}
\quad &\\
&=\sum_{k=0}^{\max\{l,m\}}
\left\{\binom{2l+2m-2k-1}{2l-1}+\binom{2l+2m-2k-1}{2m-1}\right\}
\zeta(2k)\zeta(2l+2m-2k)
\end{align*}
for $l,m\in\Zb_{\ge 1}$.
Indeed, taking $s=2l~(l=1,2,\dotsc)$
in (\ref{prop-eq:product of zeta 1})
and putting $m=l+n$,
this follows from (\ref{lem-eq:A(a,t;u)}).
\end{Rem}

\begin{proof}[Proof of Proposition $\ref{prop:representation of zetas}$]
(1)
Substituting $u=0$ in (\ref{eq:def of Z})
and using (\ref{pf-eq:harmonic product formula}), we see
\begin{align*}
Z(s,t;0)
&=\zeta(s)\zeta(t)+\cos(\pi t)\zeta(t,0;s)
+\cos(\pi s)\zeta(0,s;t)\\
&=(\cos(\pi t)-\cos(\pi s))\zeta(t,0;s)
+(1+\cos(\pi s))\zeta(s)\zeta(t)-\cos(\pi s)\zeta(s+t).
\end{align*}
Hence, the result is shown by Theorem.

(2)
Assume that $\Rep(s)>1$.
Comparing the limits of
the both sides of
(\ref{prop-eq:representation of Euler double zeta})
as $t\rightarrow \pm s+2n$,
we get
\begin{align*}
\lefteqn{(1+\cos(\pi s))\zeta(s)\zeta(\pm s+2n)}
\quad &\\
&=A(s,\pm s+2n;0)+A(\pm s+2n,s;0)
+\cos(\pi s)\zeta(s\pm s+2n)\\
&\quad
-\lim_{z\rightarrow 0}
\{\cos(\pi (z\pm s))-\cos(\pi s)\}\zeta(z\pm s+2n,0;s),
\end{align*}
where the double-signs correspond.
It is clear that the limit becomes
$0$ unless the double-signs are ``$-$''
and $n=0,1$.
In the remaining cases,
the last term of the right side is $\delta_n(s)$ because
\[
\zeta(z-s+2n,0;s)
=\frac1{s-1}\zeta(z+2n-1)+\frac{s}{12}\zeta(z+2n+1)+O(1)
\]
as $z\rightarrow 0$
(see \cite[p.425, (4.4)]{Matsumoto02}).

(3)
The result follows immediately from (\ref{eq:def of Z})
and Theorem.
\end{proof}

We next extend
the parity result \cite[Theorem 2]{Huard-Williams-Zhang96}
to the whole domain of convergence.

\begin{Prop}
\label{prop:parity result}
Let $a,b,c$ be integers such that $a+b+c$ is odd.
Assume that $a+c\ge 2$, $b+c\ge 2$ and $a+b+c\ge 3$.
If $a+b\ge 2$, then
\[
2\zeta(a,b;c)
=(-1)^a\{A(c,a;b)+A(a,c;b)\}
+(-1)^b\{A(c,b;a)+A(b,c;a)\},
\]
where every $A$-value is representable
in the form $(\ref{lem-eq:A(a,t;u)})$.
If $a+b\le 1$, then
\begin{align*}
2\zeta(a,b;c)
&=(-1)^a\{A^*(c,a;b)+A(a,c;b)\}
+(-1)^b\{A^*(c,b;a)+A(b,c;a)\}\\
&\quad
+\frac{(-1)^a 2}{(1-a-b)!}
\left.\frac{d}{ds}(s+a)_{1-a-b}\right|_{s=0}
\zeta(a+b+c-1).
\end{align*}
Here
\[
A^*(c,a;b)=
2\sum_k
\binom{a+c-2k-1}{c-2k}\zeta(2k)\zeta(a+b+c-2k),
\]
where the sum is taken over all integers
$k\in[0,c/2]$ except for $k=(a+b+c-1)/2$.
\end{Prop}

\begin{proof}
Taking the limits of the both sides of
(\ref{eq:representation of zeta by Z-function})
as $u\rightarrow c$, $t\rightarrow b$ and $s\rightarrow a$ in order,
we have
\begin{align*}
2\zeta(a,b;c)
&=(-1)^a A(a,c;b)+(-1)^b A(b,c;a)\\
&\quad
+\lim_{s\rightarrow a} ((-1)^a A(c,s;b)+(-1)^b A(c,b;s)).
\end{align*}
It is easily seen that
the limiting value equals
$(-1)^a A(c,a;b)+(-1)^b A(c,b;a)$ if $a+b\ge 2$, and
\begin{align*}
&(-1)^a A^*(c,a;b)+(-1)^b A^*(c,b;a)\\
&+\frac{(-1)^a 2}{(1-a-b)!}
\left.\frac{d}{ds}(s+a)_{1-a-b}\right|_{s=0}
\zeta(a+b+c-1)
\end{align*}
if $a+b\le 1$.
Thus, we obtain Proposition \ref{prop:parity result}.
\end{proof}

The following proposition
suggests that
$\zeta(s,t;u)$
can be represented as
a sum of products of the Riemann zeta functions,
if at least two of $s$, $t$ and $u$
are non-positive integers
in the sense of the coordinate-wise limit.

\begin{Prop}
Let $a$, $b$ and $c$ be non-negative integers
and let $s$, $t$ and $u$ be complex numbers.

$(1)$
If $s,t\ne c+1-l$ $(l=0,1,2,\dotsc)$ and $s+t\ne c+2$,
then
$\zeta(s,t;-c)=F(s,t;c)$,
where $F(s,t;c)$ is defined by
$(\ref{lem-eq:definition of F})$.

$(2)$
If $u\ne a+1-l,b+1-l,a+b+2$ $(l=0,1,2,\dotsc)$, then
\begin{align*}
\zeta(-a,-b;u)
&=
(-1)^{a+1}F(u,-a;b)+(-1)^{b+1}F(u,-b;a)\\
&\quad
+\frac{a!b!}{(a+b+1)!}\zeta(u-a-b-1)
-\delta_{a0}\zeta(u-b)-\delta_{b0}\zeta(u-a).
\end{align*}

$(3)$
For $s\in\Cb$ with $s\ne c+1-l,b+c+2$ $(l=0,1,2,\dotsc)$,
\[
\lim_{u\rightarrow -c}\zeta(s,-b;u)
=F(s,-b;c)
+\frac{(-1)^{b+1}b!c!}{(b+c+1)!}\zeta(s-b-c-1).
\]

\end{Prop}

\begin{proof}
(1) This is trivial.

(2)
We take the limits of the both sides
of (\ref{eq:representation of zeta by Z-function})
as $t\rightarrow -b$ and $s\rightarrow -a$ in order.
Then, the result is a direct consequence
of Lemma \ref{lem:evaluation of A}.

(3)
The result can be proved in a similar way to (2),
but it follows easily from
the representation
\cite[(5.3)]{Matsumoto02}
of $\zeta(s,t;u)$.
\end{proof}

\begin{Cor}
Let $a$, $b$ and $c$ be non-negative integers.

$(1)$
\[
\lim_{(s,t)\rightarrow (-a,-b)}
\zeta(s,t;-c)
=F(-a,-b;c).
\]

$(2)$
\[
\lim_{u\rightarrow -c}
\zeta(-a,-b;u)
=F(-a,-b;c)
+
\left(\frac{(-1)^{a+1}a!c!}{(a+c+1)!}
+\frac{(-1)^{b+1}b!c!}{(b+c+1)!}\right)
\zeta(-a-b-c-1).
\]

$(3)$
\[
\lim_{s\rightarrow -a}\lim_{u\rightarrow -c}\zeta(s,-b;u)
=F(-a,-b;c)
+\frac{(-1)^{b+1}b!c!}{(b+c+1)!}\zeta(-a-b-c-1).
\]

$(4)$
\[
\lim_{t\rightarrow -b}\lim_{u\rightarrow -c}\zeta(-a,t;u)
=F(-a,-b;c)
+\frac{(-1)^{a+1}a!c!}{(a+c+1)!}\zeta(-a-b-c-1).
\]
\end{Cor}

\begin{Rem}
Komori \cite{Komori08}
studied
Tornheim's double zeta values
for coordinate-wise limits
at non-positive integers
and
gave their explicit expressions
in terms of generalized Bernoulli numbers.
Our formulation seems to be more concrete than his.
\end{Rem}

To prove (2),
we have to use the following lemma
and the relation
\[
(-1)^{a+b+c}F(-a,-b;c)-\delta_{a0}\zeta(-b-c)-\delta_{b0}\zeta(-a-c)
-\delta_{a0}\delta_{b0}\delta_{c0}
=F(-a,-b;c).
\]

\begin{Lem}
Let $a$, $b$ and $c$ be non-negative integers.
Then,
\begin{align}
\lefteqn{(-1)^{a+b} F(-a,-b;c)
+(-1)^{b+c}F(-b,-c;a)+(-1)^{c+a}F(-c,-a;b)}
\quad &
\nonumber\\
&
=\left(\frac{(-1)^c a!b!}{(a+b+1)!}
+\frac{(-1)^a b!c!}{(b+c+1)!}
+\frac{(-1)^b c!a!}{(c+a+1)!}\right)\zeta(-a-b-c-1)
\nonumber\\
&\quad
+\delta_{a0}\delta_{b0}\delta_{c0}.
\label{lem-eq:convolution formula}
\end{align}
\end{Lem}

This lemma is equivalent to
Theorem 2 of Chu--Wang \cite{Chu-Wang10}.
However, their formulation is quite different from ours,
and so we now prove it for the reader's convenience.

\begin{proof}
For a non-negative integer $m$, we set
\[
\tilde{P}_m(x)
=\delta_{m0}+(-1)^m \frac{m!}{x^{m+1}}
+2^{m+1}\sum_{k=0}^\infty
(-1)^{m+k}\zeta(-m-k)\frac{(2x)^k}{k!}.
\]
We calculate the value
\[
R=2^{-a-b-c-2}[x^{-1}]
(\tilde{P}_a(x)-\delta_{a0})
(\tilde{P}_b(x)-\delta_{b0})
(\tilde{P}_c(x)-\delta_{c0})
\]
in two ways, where
$[x^{-1}]f(x)$ denotes
the formal residue of a formal Laurent series $f(x)$.
We first use the definition of $\tilde{P}_m(x)$ to obtain
\begin{align*}
R&=
(-1)^{a+b} F(-a,-b;c)
+(-1)^{b+c}F(-b,-c;a)+(-1)^{c+a}F(-c,-a;b)\\
&\quad
-\left(\frac{(-1)^c a!b!}{(a+b+1)!}
+\frac{(-1)^a b!c!}{(b+c+1)!}
+\frac{(-1)^b c!a!}{(c+a+1)!}\right)\zeta(-a-b-c-1).
\end{align*}
We next apply Proposition 3.1 in \cite{Onodera}
to get
$R=\delta_{a0}\delta_{b0}\delta_{c0}$.
Thus, we have (\ref{lem-eq:convolution formula}).
\end{proof}

We finally show
the behavior of $\zeta_{SU(3)}(s)$
at each integer.

\begin{Prop}
\label{prop:Witten zeta values}
Let $a$ be a positive integer.

$(1)$
\cite[Theorem 3]{Huard-Williams-Zhang96}
\[
\zeta_{SU(3)}(a)
=\frac{2^{a+2}}{1+(-1)^a 2}
\sum_{k=0}^{[a/2]}
\binom{2a-2k-1}{a-1}
\zeta(2k)\zeta(3a-k).
\]

$(2)$
$\zeta_{SU(3)}(0)=1/3$
and $\zeta_{SU(3)}'(0)=\log (2^{4/3}\pi)$.

$(3)$
If $a$ is odd, then
$\zeta_{SU(3)}(s)$ has a simple zero at $s=-a$,
and
\begin{equation}
\label{prop-eq:SU(3)-zeta at odd}
\zeta_{SU(3)}'(-a)=
2^{-a+2}\sum_{k=0}^{(a-1)/2}
\binom{a}{2k}
\zeta(-a-2k)\zeta'(-2a+2k)
+\frac{2^{-a+1}(a!)^2}{(2a+1)!}\zeta'(-3a-1).
\end{equation}
In particular, $\sign(\zeta_{SU(3)}'(-a))=(-1)^{(a-1)/2}$.

$(4)$
If $a$ is even, then
$\zeta_{SU(3)}(s)$ has a zero of order two at $s=-a$,
and
\begin{equation}
\label{prop-eq:SU(3)-zeta at even}
\zeta_{SU(3)}''(-a)
=2^{-a+2}\sum_{k=0}^{a/2}
\binom{a}{2k}
\zeta'(-a-2k)\zeta'(-2a+2k).
\end{equation}
In particular, $\sign(\zeta_{SU(3)}''(-a))=(-1)^{a/2}$.
\end{Prop}

\begin{Rem}
The value of Witten's zeta function $\zeta_G(s)$
of each finite group $G$
at $s=-2$ coincides with the order of $G$.
In this viewpoint,
it is attractive to clarify the behavior
of $\zeta_G(s)$ at $s=-2$
in the case of $G$ being an infinite compact topological group.
In \cite{Kurokawa-Ochiai},
Kurokawa and Ochiai
studied the values of
Witten's zeta functions
at negative integers,
and proved
that
$\zeta_{SU(3)}(s)$ has a zero at each negative integer.
Proposition \ref{prop:Witten zeta values}
can be regarded as
a refinement of their result.
Moreover, as seen below,
our proof
reveals
that
a zero of $\zeta_{SU(3)}(s)$ at each negative integer
comes from 
the gamma factors appearing
in the left sides of
(\ref{lem-eq1:GPFD1}) and (\ref{lem-eq1:GPFD2}).
\end{Rem}

\begin{proof}[Proof of Proposition $\ref{prop:Witten zeta values}$]
We use here the integral representation
(\ref{prop-eq:integral representation of Witten zeta})
of $\zeta_{SU(3)}(s)$.

(1) The result is clear from (\ref{lem-eq:A(a,t;u)}).

(2)
By (\ref{eq:A-function shifted}),
we see that,
if $K$ is a non-negative integer
and $-K/2+1/4<\Rep(s)<K+1/2$, then
\begin{align*}
\lefteqn{2^{-s-1}(1+2\cos(\pi s))\zeta_{SU(3)}(s)}
\quad &
\nonumber\\
&=2\sum_{k=0}^{K}
\frac{(s)_k}{k!}\cos^2\left(\frac{\pi (s-k)}2\right)
\zeta(s-k)\zeta(2s+k)
+\frac{\sin(\pi s)}{\Gamma(s)}
R_K(s),
\label{pf-eq:Witten zeta}
\end{align*}
where $R_K(s)$ is a holomorphic function.
We note that every term except for the term with $k=0$
has a zero of order at least two
at $s=0$. Hence, the values at $s=0$ can be immediately calculated.

(3) Set $K=2a+1$.
If $a$ is odd, then
the terms with $k=0,1,\dotsc,a$ satisfying $k\equiv a \pmod 2$
and the term with $k=2a+1$ have a simple zero at $s=-a$,
and the others have a zero of order at least two.
Hence, we can easily obtain the first part of the result.
The last part follows from
the functional equation of the Riemann zeta function.
Indeed,
we can show that
the sign of each term
on the right side of (\ref{prop-eq:SU(3)-zeta at odd})
coincides with $(-1)^{(a-1)/2}$.

(4)
Put $K=2a+1$ again.
In the same way, we see that
$\zeta_{SU(3)}(s)$ has a zero of order at least two at
$s=-a$ if $a$ is even.
In order to determine the multiplicity of the zero,
we now show (\ref{prop-eq:SU(3)-zeta at even}).
Assume that $0<\varepsilon<1/2$.
Set
\[
f(s,\eta)
=\frac{\sin(\pi s)}{\Gamma(s)}
\cot\left(\frac{\pi (s-\eta)}2\right)
\Gamma(s+\eta)\Gamma(-\eta)
\zeta(s-\eta)\zeta(2s+\eta).
\]
Then, by shifting the contour,
we obtain the following expression
of $\zeta_{SU(3)}(s)$ which is valid around $s=-a$:
\[
2^{-s-1}(1+2\cos(\pi s))\zeta_{SU(3)}(s)
=
-\sum_{k=0}^{a/2} U_k(s)
+\sum_{l=0}^{a/2-1} V_l(s)
+W(s)
+I(s),
\]
where
$
U_k(s)
=
\Res_{\eta=k} f(s,\eta)$,
$V_l(s)
=
\Res_{\eta=-s-l} f(s,\eta)$,
$W(s)
=\Res_{\eta=1-2s}f(s,\eta)$
and
\[
I(s)=
\frac1{2\pi i}\int_{C_\varepsilon} f(s,\eta)d\eta
\]
whose
contour $C_\varepsilon$ describes the union of
$C_\varepsilon^{(1)}:a/2-i\infty \rightarrow a/2-i\varepsilon$,
$C_\varepsilon^{(2)}:a/2+\varepsilon e^{i\theta}
(\theta:-\pi/2 \rightarrow \pi/2)$
and
$C_\varepsilon^{(3)}:a/2+i\varepsilon \rightarrow a/2+i\infty$.
We here remark that
the poles at $\eta=k,s-2m,-s-a/2-m~(k=0,1,\dotsc,a/2;m=0,1,2,\dotsc)$
lie on the left of the contour $C_\varepsilon$
and the poles at $\eta=1-2s,-s-l,a/2+n
~(l=0,1,\dotsc,a/2-1;n=1,2,\dotsc)$
lie on the right.
Hence,
\[
2^{a-1} 3\,\zeta_{SU(3)}''(-a)
=
-\sum_{k=0}^{a/2} U_k''(-a)
+\sum_{l=0}^{a/2-1} V_l''(-a)
+W''(-a)
+I''(-a).
\]
By simple calculation, we first see that,
for $k=0,1,\dotsc,a/2$
and
$l=0,1,\dotsc,a/2-1$,
\begin{align*}
U_k''(-a)
&=
\begin{cases}
\displaystyle
-8 \binom{a}{k}
\zeta'(-a-k)\zeta'(-2a+k)
& \text{if $k$ is even},\\
\displaystyle
\pi^2 \binom{a}{k}
\zeta(-a-k)\zeta(-2a+k)
& \text{if $k$ is odd},
\end{cases}
\\
V_l''(-a)
&=
\begin{cases}
\displaystyle
4 \binom{a}{l}
\zeta'(-a-l)\zeta'(-2a+l)
& \text{if $l$ is even},\\
\displaystyle
-2\pi^2 \binom{a}{l}
\zeta(-a-l)\zeta(-2a+l)
& \text{if $l$ is odd},
\end{cases}
\end{align*}
and
\[
W''(-a)=\frac{3\pi^2}{2}
\frac{(a!)^2}{(2a+1)!}\zeta(-3a-1).
\]
We next evaluate
\[
I''(-a)
=\frac1{2\pi i}\int_{C_{\varepsilon}}f''(-a,\eta)d\eta,
\]
where $f''$ means $(\partial/\partial s)^2 f$.
Since
the integrals on $C_\varepsilon^{(1)}$ and $C_\varepsilon^{(3)}$
cancel each other out,
we obtain
\begin{align*}
I''(-a)
&=\frac1{2\pi i}\int_{C_\varepsilon^{(2)}}
\Res_{\eta=a/2}f''(-a,\eta)
\frac{d\eta}{\eta-a/2}\\
&\quad
+\frac1{2\pi i}\int_{C_\varepsilon^{(2)}}
\left(f''(-a,\eta)-
\Res_{\eta=a/2}f''(-a,\eta)\cdot
\frac1{\eta-a/2}
\right)d\eta.
\end{align*}
We note that the integrand in the second integral
is holomorphic at $\eta=a/2$,
and so the second integral tends to zero
as $\varepsilon$ tends to zero.
Since $I(s)$ is independent of the choice of $\varepsilon$,
we get
\[
I''(-a)=\frac12 \Res_{\eta=a/2} f''(-a,\eta)
=D_1(a)+D_2(a),
\]
where
\begin{align*}
D_1(a)
&=
\begin{cases}
\displaystyle
-\frac{\pi^2}2\binom{a}{a/2}\zeta(-3a/2)^2
& \text{if $a\equiv 2 \pmod 4$},\\
0 & \text{if $a\equiv 0 \pmod 4$},
\end{cases}
\\
D_2(a)
&=
\begin{cases}
0 & \text{if $a\equiv 2 \pmod 4$},\\
\displaystyle
-2 \binom{a}{a/2}
\zeta'(-3a/2)^2
& \text{if $a\equiv 0 \pmod 4$}.
\end{cases}
\end{align*}
Combining the above results,
we have
\begin{align*}
-
\lefteqn{\sum_{\substack{0\le k\le a/2\\ k:\text{odd}}}
U_k''(-a)
+\sum_{\substack{0\le l\le a/2-1\\ l:\text{odd}}}
V_l''(-a)
+W''(-a)+D_1(a)}
\quad &\\
&=
-\frac{3\pi^2}{2}
\sum_{k=0}^a \binom{a}{k}
\zeta(-a-k)\zeta(-2a+k)
+\frac{3\pi^2}{2}
\frac{(a!)^2}{(2a+1)!}\zeta(-3a-1)\\
&=0,
\end{align*}
where in the last step we have used
(\ref{lem-eq:convolution formula})
with $a=b=c$.
Moreover, we see
\begin{align*}
-\lefteqn{\sum_{\substack{0\le k\le a/2\\ k:\text{even}}}
U_k''(-a)
+\sum_{\substack{0\le l\le a/2-1\\ l:\text{even}}}
V_l''(-a)
+D_2(a)}
\quad &\\
&=
6\sum_{k=0}^{a/2}
\binom{a}{2k}
\zeta'(-a-2k)\zeta'(-2a+2k).
\end{align*}
Thus, we obtain (\ref{prop-eq:SU(3)-zeta at even}).
In the same way as (3) above,
we get
$\sign(\zeta_{SU(3)}''(-a))=(-1)^{a/2}$,
which completes the proof of Proposition
\ref{prop:Witten zeta values}.
\end{proof}


\appendix
\section{A functional equation for Euler's double zeta function}
The $A$-function
(\ref{thm-eq:definition of A})
has not been found in previous papers
on multiple zeta functions.
However,
as seen in the next proposition,
$A(s,t;0)$
is related to
the functional equation
of $\zeta(s,t)=\zeta(t,0;s)$
which was obtained by Matsumoto
\cite[Theorem 1]{Matsumoto04}.

\begin{Prop}
\label{prop:functional equation}
Set
\[
h(s,t)=\zeta(s,t)-\frac{\Gamma(1-t)}{\Gamma(s)}
\Gamma(s+t-1)\zeta(s+t-1).
\]
Then, we have
\begin{align}
\frac{h(s,t)}{(2\pi)^{s+t-1}\Gamma(1-t)}
&=\cos\left(\frac{\pi}2(s+t-1)\right)
\frac{h(1-t,1-s)}{\Gamma(s)}
\nonumber\\
&\quad
+\sin\left(\frac{\pi}2(s+t-1)\right)
\frac{\Gamma(1-s)}{\pi}A(1-s,1-t;0).
\label{prop-eq:functional equation}
\end{align}
In particular,
the second term on the right side of
{\upshape (\ref{prop-eq:functional equation})}
vanishes
on the hyperplane 
$s+t=2k+1$ $(k\in\Zb\setminus\{0\})$
$($cf.\ \cite[Theorem 2.2]{Komori-Matsumoto-Tsumura10}$)$.
\end{Prop}

\begin{Rem}
Firstly, the function $g(u,v)$ in Matsumoto's paper
coincides with $h(v,u)$.
Secondly,
it may seem that
the singularities of $\zeta(s,t)$ are located on
the hyperplanes $s=1-l$ and $s+t=2-l~(l=0,1,2,\dotsc)$.
However,
the singularities on
$s=-l$ and $s+t=-1-2l~(l=0,1,2,\dotsc)$
are fake,
namely, the singularities of $\zeta(s,t)$
are only located on the hyperplanes
$s=1$, $s+t=1$ and $s+t=2-2l~(l=0,1,2,\dotsc)$.
This can be confirmed,
for instance, 
by
(\ref{eq:A-function shifted})
and
(\ref{prop-eq:representation of Euler double zeta}).
Hence, the last part of the proposition is justified.
\end{Rem}

\begin{proof}[Proof of Proposition {\upshape \ref{prop:functional equation}}]
We first recall the usual integral representation
of $\zeta(s,t)$
(cf.\ \cite[(5.2)]{Matsumoto02}):
\[
\zeta(s,t)=
\frac1{2\pi i}\int_{(c)}
\frac{\Gamma(s+\eta)\Gamma(-\eta)}{\Gamma(s)}
\zeta(t-\eta)\zeta(s+\eta)d\eta
\]
for $s,t\in\Cb$ with $\Rep(s)>1$ and $\Rep(t)>1$,
where
$-\Rep(s)+1<c<0$
and
the contour $(c)$
describes the line from $c-i\infty$ to $c+i\infty$.
Since the residue of the integrand at $\eta=t-1$
is
\[
-\frac{\Gamma(1-t)}{\Gamma(s)}\Gamma(s+t-1)\zeta(s+t-1)
\]
unless $t=1,2,3,\dotsc$,
we shift the contour to obtain
\[
h(s,t)=
\frac1{2\pi i}\int_{C}
\frac{\Gamma(s+\eta)\Gamma(-\eta)}{\Gamma(s)}
\zeta(t-\eta)\zeta(s+\eta)d\eta
\]
for $s,t\in\Cb$ with
$s\ne1-k$ and $t\ne k$ ($k=0,1,2,\dotsc$),
where the contour $C$
is a line from $-i\infty$ to $i\infty$
indented in such a manner as to separate
the points at $\eta=-s+1-l, t-l~(l=0,1,2,\dotsc)$
from the points at $\eta=0,1,2,\dotsc$.
By the functional equation of the Riemann zeta function,
the integrand is equal to
$(2\pi)^{s+t-1}\Gamma(1-t)$ times
\begin{align*}
&
\cos\left(\frac{\pi}2(s+t-1)\right)
\frac{\Gamma(1-t+\eta)\Gamma(-\eta)}{\Gamma(s)\Gamma(1-t)}
\zeta(1-s-\eta)\zeta(1-t+\eta)
\\
&
+\sin\left(\frac{\pi}2(s+t-1)\right)
\frac{\Gamma(1-s)}{\pi}
\sin(\pi (1-s))\\
&\quad
\times
\cot\left(\frac{\pi}2(1-s-\eta)\right)
\frac{\Gamma(1-t+\eta)\Gamma(-\eta)}{\Gamma(1-t)}
\zeta(1-s-\eta)\zeta(1-t+\eta).
\end{align*}
Thus, we obtain (\ref{prop-eq:functional equation}).
\end{proof}

We now compare our result with the result of Matsumoto
to obtain a new representation
of $A(s,t;0)$.
For $(s,t)\in\Cb^2$ with $\Rep(s)<0$ and $\Rep(t)>1$,
set
\[
F_{\pm}(s,t)
=\sum_{k=1}^\infty \sigma_{s+t-1}(k)\Psi(t,s+t;\pm 2\pi i k),
\]
where
$\sigma_\nu(k)=\sum_{d|k}d^{\nu}$
and
$\Psi(\alpha,\gamma;z)$
is the confluent hypergeometric function
of the second kind.
It is known that
$F_{\pm}(s,t)$ can be continued meromorphically
to the whole space $\Cb^2$.

\begin{Cor}
The function
$A(s,t;0)$ can be represented in terms of the $F_{\pm}$-functions:
\[
2\Gamma(s)A(s,t;0)
=(2\pi i)^{s+t} F_+(s,t)+(-2\pi i)^{s+t}F_-(s,t).
\]
\end{Cor}

\begin{proof}
Propositions 1 and 2 in \cite{Matsumoto04}
show
\[
\frac{h(s,t)}{(2\pi)^{s+t-1}\Gamma(1-t)}
=e^{\pi i(s+t-1)/2}F_+(t,s)+e^{\pi i(1-s-t)/2}F_-(t,s)
\]
and
\begin{equation}
\label{pf-eq:functional equation of F}
F_{\pm}(1-t,1-s)=(\pm 2\pi i)^{s+t-1}F_{\pm}(s,t),
\end{equation}
respectively.
These suggest
\[
\frac{h(1-t,1-s)}{\Gamma(s)}
=F_+(t,s)+F_-(t,s)
\]
and so we see
\begin{align*}
\frac{h(s,t)}{(2\pi)^{s+t-1}\Gamma(1-t)}
&=\cos\left(\frac{\pi}2(s+t-1)\right)
\frac{h(1-t,1-s)}{\Gamma(s)}\\
&\quad
+i\sin\left(\frac{\pi}2(s+t-1)\right)
(F_+(t,s)-F_-(t,s)).
\end{align*}
By comparing this with (\ref{prop-eq:functional equation}),
we have
$\Gamma(1-s)A(1-s,1-t;0)=\pi i(F_+(t,s)-F_-(t,s))$.
Thus, we use (\ref{pf-eq:functional equation of F})
to obtain the result.
\end{proof}


\end{document}